\documentclass[11pt]{amsproc}
 \usepackage[margin=1in]{geometry}
\usepackage{setspace,fullpage}
\usepackage{graphicx}
\usepackage[nice]{nicefrac}
\usepackage{amssymb}
\usepackage{amsmath,amsthm,amscd,amssymb, mathrsfs}

\usepackage{latexsym}

\numberwithin{equation}{section}

\theoremstyle{plain}
\newtheorem{theorem}{Theorem}[section]
\newtheorem{lemma}[theorem]{Lemma}

\newtheorem{conjecture}[theorem]{Conjecture}

\theoremstyle{definition}

\theoremstyle{remark}
\newtheorem{remark}[theorem]{Remark}

\newtheorem{case[theorem]}{Case}

\title[\parbox{14cm}{\centering{ Sharp endpoint estimates for the $X$-ray transform and the Radon transform in finite fields \hspace{1in}}} \quad]{Sharp endpoint estimates for the $X$-ray transform and the Radon transform in finite fields }
\author{ Doowon Koh }

\address{Department of Mathematics\\
Chungbuk National University \\
Cheongju city, Chungbuk-Do 361-763 Korea}
\email{koh131@chungbuk.ac.kr}

\thanks{
This research was supported by Basic Science Research Program through the National Research Foundation of Korea(NRF) funded by the Ministry of Education, Science and Technology(2012010487)
}


\begin{document}

\begin{abstract}This note establishes sharp $L^p-L^r$ estimates for $X$-ray transforms and Radon transforms in finite fields.
\end{abstract}

\maketitle

\section{Introduction}
Let $\mathbb F_q^d, d\geq 2,$ be a $d$-dimensional vector space over the finite field $\mathbb F_q$ with $q$ elements. We endow $\mathbb F_q^d$ with a normalized counting measure $dx.$  For each $q$, we denote by $M_q$  a collection of certain subsets of $\mathbb F_q^d.$  Recall that a normalized surface measure $d\sigma$ on $M_q$ can be defined by the relation
$$ \int\Omega(w) d\sigma(w) = \frac{1}{|M_q|} \sum_{w\in M_q} \Omega(w)$$
where $|M_q|$ denotes the cardinality of $M_q$ and $\Omega$ is a complex-valued function on $M_q.$ 
For any complex-valued function $f$ on $(\mathbb F_q^d,dx)$ and $w\in M_q$, we consider an operator $T_{M_q}$ defined by
$$ T_{M_q}f(w)=\frac{1}{|w|} \sum_{x\in w} f(x).$$

We are interested in determining  exponents $1\leq p, r\leq \infty$ such that the following inequality holds:
$$ \|T_{M_q}f\|_{L^r(M_q,d\sigma)} \lesssim \|f\|_{L^p(\mathbb F_q^d,dx)} \quad\mbox{for all}~~ f ~~\mbox{on}~~\mathbb F_q^d,$$
where the operator norm of $T_{M_q}$ is independent of $q$, the size of the underlying finite field $\mathbb F_q.$
If  $M_q$  is $\Pi_k,$ a collection of all $k$-planes in $\mathbb F_q^d$ with $1\leq k\leq d-1$, then the operator $T_{\Pi_k}$ is called as the $k$-plane transform. In particular, $T_{\Pi_1}$ and $T_{\Pi_{(d-1)}}$ are known as the $X$-ray transform and the Radon transform respectively. In Euclidean space,  the complete mapping properties of $k$-plane transforms were proved by M. Christ in \cite{Ch84}. Readers may refer to \cite{So79}, \cite{OS82}, \cite{Dr84}  for the description of  $k$-plane transforms in the Euclidean setting.
In 2008,  Carbery, Stones, and Wright \cite{CSW08} initially studied the mapping properties of $k$-plane transforms in finite fields. Using combinatorial arguments, they proved the following theorem. 
\begin{theorem}\label{Car} Let $d\geq 2$ and $1\leq k\leq d-1$ be an integer. If 
\begin{equation}\label{8}\|T_{\Pi_k}f\|_{L^r(\Pi_k, d\sigma)} \leq C \|f\|_{L^p(\mathbb F_q^d,dx)}\end{equation}
holds with $C$ independent of $|\mathbb F_q|$, then $(1/p, 1/r)$ lies in the convex hull $H$ of 
$$((k+1)/(d+1), 1/(d+1)), (0,0),(1,1)~~\mbox{and}~~(0,1).$$ 
Conversely, if $(1/p, 1/r)$ lies in $H\setminus\left\{ ((k+1)/(d+1), 1/(d+1))\right\},$ then (\ref{8}) holds with $C$ independent of $|\mathbb F_q|.$ Finally, if 
$(1/p,1/r)= ((k+1)/(d+1), 1/(d+1))$, then the restricted type inequality 
\begin{equation}\label{9} \|T_{\Pi_k}f\|_{L^{d+1}( \Pi_k,d\sigma)} \leq C \|f\|_{L^{\frac{d+1}{k+1}, 1}(\mathbb F_q^d,dx)}\end{equation}
holds with $C$ independent of $|\mathbb F_q|.$\end{theorem}

Notice from Theorem \ref{Car} that if one could show that the restricted type inequality (\ref{9}) can be replaced by the strong type inequality, then the mapping properties of the $k$-plane transforms in finite fields would be completely established. Namely, our task would prove the following conjecture.
\begin{conjecture}\label{Radonmain} Let $d\geq 2$ and $1\leq k\leq (d-1)$ be integers. Then, we have 
$$ \|T_{\Pi_k}f\|_{L^{d+1}( \Pi_k, d\sigma)} \leq C \|f\|_{L^{\frac{d+1}{k+1}}(\mathbb F_q^d,dx)} \quad\mbox{for all}~~f ~~\mbox{on}~~\mathbb F_q^d,$$
where $C$ is independent of $|\mathbb F_q|.$ \end{conjecture}

\subsection{Statement of main results}
In this paper we prove that Conjecture \ref{Radonmain} is true for the $X$-ray transform and the Radon transform.  More precisely, we obtain the following theorem.
\begin{theorem}\label{newmain1} Let $d\geq 2$ be any integer. If $k=1$ or $k=d-1,$ then  
$$ \|T_{\Pi_k}f\|_{L^{d+1}( \Pi_k, d\sigma)} \leq C \|f\|_{L^{\frac{d+1}{k+1}}(\mathbb F_q^d,dx)} \quad\mbox{for all}~~f ~~\mbox{on}~~\mathbb F_q^d,$$
where $C$ is independent of $|\mathbb F_q|.$ \end{theorem}
In order to prove Theorem \ref{newmain1} for the $X$-ray transform ($k=1$),  we shall adapt both the combinatorial arguments  in \cite{CSW08} and the skills in \cite{LL10} for endpoint estimates. On the other hand,   a Fourier analytic argument will be required to prove  Theorem \ref{newmain1} for the Radon transform ($k=d-1$) .
\begin{remark} After writing this paper, the author realized that  our result for the $X$-ray transform is a corollary of Theorem 1.1 in the paper \cite {EOT}.
This was pointed out by R. Oberlin. 
\end{remark}

\section{Proof of the mapping properties of the $X$-ray transform}
In this section, we restate and  prove Theorem \ref{newmain1} in the case of the  $X$-ray transform.
Namely, we prove the following statement which implies the sharp boundedness of the $X$-ray transform.
\begin{theorem}\label{newmain11} Let $d\geq 2$ be any integer.  
$$ \|T_{\Pi_1}f\|_{L^{d+1}( \Pi_1, d\sigma)} \leq C \|f\|_{L^{\frac{d+1}{2}}(\mathbb F_q^d,dx)} \quad\mbox{for all}~~f ~~\mbox{on}~~\mathbb F_q^d,$$
where $C$ is independent of $|\mathbb F_q|.$ \end{theorem}
\begin{proof}
 We begin by following the argument in \cite{LL10}. Without loss of generality, we may assume that
$f$ is a non-negative real-valued function and 
\begin{equation}\label{easy1} \sum_{x\in \mathbb F_q^d} f(x)^{\frac{d+1}{2}}=1.\end{equation}
Thus, it is natural to assume that $\|f\|_\infty \leq 1.$ Furthermore, we may assume that $f$ is written by a step function
\begin{equation}\label{easy2} f(x)=\sum_{i=0}^\infty 2^{-i} E_i(x),\end{equation}
where the sets $E_i$ are disjoint subsets of $\mathbb F_q^d$, and  here, and throughout the paper, we write $E(x)$ for the characteristic function on a set $E\subset \mathbb F_q^d.$
From (\ref{easy1}) and (\ref{easy2}), we also assume that
\begin{equation} \label{easy3} \sum_{j=0}^\infty 2^{-\frac{(d+1)j}{2}}|E_j|=1\quad\mbox{and } ~~ |E_j|\leq 2^{\frac{(d+1)j}{2}}~~\mbox{for all}~~j=0,1,\cdots.\end{equation}
Since $dx$ is the normalized counting measure on $\mathbb F_q^d,$  the assumption (\ref{easy1}) shows that we only need to prove 
\begin{equation}\label{suff0} \|T_{\Pi_1}f\|^{d+1}_{L^{d+1}( \Pi_1, d\sigma)} \lesssim q^{-2d},\end{equation}
where $f$ satisfies  (\ref{easy2}) and (\ref{easy3}). Since we have assumed that $f\geq 0$, it is clear that $T_{\Pi_1}f$ is also a non-negative real-valued function on $\Pi_1.$ By expanding the left-hand side of (\ref{suff0}) and using the facts that $|w|=q$ for $w\in \Pi_1$ and $|\Pi_1|\sim q^{2(d-1)}$, we see that
$$ \|T_{\Pi_1}f\|^{d+1}_{L^{d+1}( \Pi_1, d\sigma)} = \frac{1}{|\Pi_1|} \sum_{w\in \Pi_1} \left(T_{\Pi_1}f(w)\right)^{d+1} $$
$$\sim \frac{1}{q^{d+1}} \frac{1}{q^{2(d-1)}} \sum_{i_0=0}^\infty  \dots \sum_{i_d=0}^\infty 2^{-(i_0+\dots +i_{d})} \sum_{(x_0,\dots,x_{d})\in E_{i_{0}}\times \dots \times E_{i_{d}}} \sum_{w\in \Pi_1} w(x_0)\dots w(x_{d})$$
$$\sim \frac{1}{q^{d+1}} \frac{1}{q^{2(d-1)}} \sum_{0=i_0\leq i_1\leq \dots\leq i_{d}<\infty}   2^{-(i_0+\dots +i_{d})} \sum_{(x_0,\dots,x_{d})\in E_{i_{0}}\times \dots \times E_{i_{d}}} \sum_{w\in \Pi_1} w(x_0)\dots w(x_{d}),$$
where the last line follows from  the symmetry of $i_0, \cdots,i_{d}.$ We now follows the argument in \cite{CSW08}.
Notice that we can write
$$\sum_{(x_0,\dots,x_{d})\in E_{i_{0}}\times \dots \times E_{i_{d}}}=\sum_{s=0}^\infty \sum_{(x_0,\dots,x_{d})\in \Delta(s,i_0,\dots,i_{d})},$$
where $\Delta(s, i_0,\dots,i_{d})=\{(x_0,\dots,x_{d})\in E_{i_{0}}\times \dots \times E_{i_{d}}: [x_0,\dots,x_{d}]~\mbox{is a}~s\mbox{-plane}\}$ and $[x_0,\dots,x_{d}]$ denotes the smallest affine subspace containing the elements $x_0,\dots,x_{d}.$  In addition, observe that  if $ s>1$ and  $(x_0,\dots,x_{d}) \in \Delta(s,i_0,\dots,i_{d}),$  then the sum over $w\in \Pi_1$ vanishes. On the other hand, if $s=0,1,$ then the sum over $w\in \Pi_1$
is same as the number of lines containing the unique $s$-plane, that is $\sim q^{(d-1)(1-s)}.$ From these observations and (\ref{suff0}),  it is enough to prove that for all $E_i, i=0,1,\dots,$ satisfying  (\ref{easy3}), 
\begin{equation}\label{suffmain} \sum_{i_0=0}^\infty \sum_{i_1\geq i_0}^\infty \cdots \sum_{i_{d}\geq i_{d-1}}^\infty 2^{-(i_0+i_1+\cdots +i_{d})} \sum_{s=0}^1 |\Delta(s,i_0,\dots,i_{d})| q^{-s(d-1)} \lesssim 1.
\end{equation}

Namely, it suffices to prove that for every $d\geq 2$ and $s=0,1$,
 $$\mbox{A}=\sum_{i_0=0}^\infty \sum_{i_1\geq i_0}^\infty \cdots \sum_{i_{d}\geq i_{d-1}}^\infty 2^{-(i_0+i_1+\cdots +i_{d})}  |\Delta(s,i_0,\dots,i_{d})| q^{-s(d-1)} \lesssim 1,$$
 where $\Delta(s,i_0,\dots,i_{d})$ is defined as before, and the sets $E_i$, i=0,1,\dots,  satisfy  (\ref{easy3}). Suppose that $s=0.$ Since $|\Delta(0,i_0,\dots,i_{d})|\leq |E_{i_0}|$, it follows
 $$ \mbox{A}\leq \sum_{i_0=0}^\infty \sum_{i_1\geq i_0}^\infty \cdots \sum_{i_{d}\geq i_{d-1}}^\infty 2^{-(i_0+i_1+\cdots +i_{d})} |E_{i_0}|. $$
 Since the sum of a convergent geometric series is similar to the value of the first term,  we have the desirable conclusion for $s=0$:  
 $$\mbox{A}\lesssim\sum_{i_0=0}^\infty |E_{i_0}| 2^{-(d+1)i_0} \leq \sum_{i_0=0}^\infty |E_{i_0}| 2^{-\frac{(d+1)i_0}{2}} =1,$$
 where the last equality is obtained from (\ref{easy3}).
 
 Next, we assume that $s=1.$ We must show that for all $E_i, i=0,1,\dots$,  satisfying  (\ref{easy3}), we have
 \begin{equation}\label{aim1}\sum_{i_0=0}^\infty \sum_{i_1\geq i_0}^\infty \cdots \sum_{i_{d}\geq i_{d-1}}^\infty 2^{-(i_0+i_1+\cdots +i_{d})}  |\Delta(1,i_0,\dots,i_{d})| q^{-(d-1)} \lesssim 1.\end{equation}
 We estimate the upper bound of $|\Delta(1,i_0,\dots,i_{d})|.$ Fix $x_{i_0} \in E_{i_0}$ which has $|E_{i_0}|$ choices.
 Notice that if $(x_{i_0}, \dots, x_{i_d}) \in \Delta(1,i_0,\dots,i_{d}),$ then all points $x_{i_0}, \dots, x_{i_{d}}$ must lie on a line, which is determined by at least two different points of them . Therefore, for each $l=1,2, \dots,d,$, we can define 
 $$ L(l)=\{(x_{i_0}, \dots, x_{i_{d}}) \in \Delta(1,i_0, \dots, i_{d}): [x_{i_0},\dots, x_{i_l}]~\mbox{is a line, and}~[x_{i_0},\dots, x_{i_{l-1}}]~\mbox{is a point}\},$$
 where we recall that $[\alpha_1,\dots, \alpha_s]$ means the smallest affine subspace containing all points $\alpha_1,\dots, \alpha_s$ in $\mathbb F_q^d.$  It is clear that $\Delta(1,i_0,\dots,i_{d})=\cup_{l=1}^d L(l),$ which implies that
 \begin{equation}\label{aim2} |\Delta(1,i_0,\dots,i_{d})|\leq \sum_{l=1}^d |L(l)|.\end{equation}
 By the definition of $L(l), l=1,\dots,d,$ it follows that for every $l=1,\dots,d,$
 \begin{equation}\label{aim3}|L(l)|\leq |E_{i_0}||E_{i_{l}}|q^{d-l}.\end{equation} 
 To see this, first fix $x_{i_0}\in E_{i_0}$ which has $|E_{i_0}|$ choices. For each fixed $x_{i_0}\in E_{i_0},$ if $(x_{i_0},\dots,x_{d})\in L(l)$, then all points $x_{i_1},\dots, x_{i_{l-1}}$ are automatically chosen as $x_{i_0},$ and there are at most $|E_{i_{l}}|$ choices for $x_{i_{l}}\in E_{i_l}.$ Since $x_0$ and $x_{i_{l}}$ determine a fixed line, all points $x_{i_{l+1}},\dots, x_{i_{d}}$ must lie on the line. Thus, there are at most $q$ choices for each $x_{i_{l+1}},\dots, x_{i_{d}}$, because a line contains exactly $q$ points in $\mathbb F_q^d.$

From (\ref{aim1}), (\ref{aim2}), and (\ref{aim3}), it suffices to prove that for every $l=1,\dots,d,$
$$\mbox{B}=\sum_{i_0=0}^\infty \sum_{i_1\geq i_0}^\infty \cdots \sum_{i_{d}\geq i_{d-1}}^\infty 2^{-(i_0+i_1+\cdots +i_{d})} |E_{i_0}||E_{i_{l}}|q^{1-l} \lesssim 1.$$
Since $|E_{i_{l}}|\leq q^d$ and $l\geq 1,$ it is easy to see that $|E_{i_{l}}|^{(l-1)/d} q^{1-l} \lesssim 1.$ Therefore, it follows that
$$ \mbox{B}\lesssim \sum_{i_0=0}^\infty \sum_{i_1\geq i_0}^\infty \cdots \sum_{i_{d}\geq i_{d-1}}^\infty 2^{-(i_0+i_1+\cdots +i_{d})}|E_{i_0}| |E_{i_{l}}|^{\frac{d+1-l}{d}}.$$
Since $\frac{d+1-l}{d}>0$, applying (\ref{easy3})  gives
$$ \mbox{B}\lesssim \sum_{i_{0}=0}^\infty \sum_{i_{1}\geq i_{0}}^\infty \cdots \sum_{i_{d}\geq i_{d-1}}^\infty 2^{-(i_{0}+i_{1}+\cdots +i_{d})}|E_{i_{0}}| 2^{\frac{(d+1-l)(d+1)i_{l}}{2d}}.$$
Compute  the inner summations by checking  that each of them is a convergent geometric series. It follows that
$$ \mbox{B}\lesssim \sum_{i_{0}=0}^\infty |E_{i_0}| 2^{\frac{(-d^2-dl-l+1)i_{0}}{2d}} \leq  \sum_{i_{0}=0}^\infty |E_{i_{0}}| 2^{-\frac{(d+1)i_0}{2}}=1,$$
where the last equality follows from (\ref{easy3}).  Thus, we complete the proof of Theorem \ref{newmain11}.
\end{proof}
\begin{remark} It seems that the similar arguments as above  work for settling Conjecture \ref{Radonmain}, but   it may not be simple to estimate  $|\Delta(s,i_0,\dots,i_{d})|.$ 
\end{remark}
\section{Proof of mapping properties of the Radon transform}
In this section, we prove Theorem \ref{newmain1} in the case of the Radon transform. Namely, we shall prove the following.  

\begin{theorem}\label{newradon} Let $d\geq 2$ be any integer. Then,  
$$ \|T_{\Pi_{d-1}}f\|_{L^{d+1}( \Pi_{d-1}, d\sigma)} \leq C \|f\|_{L^{\frac{d+1}{d}}(\mathbb F_q^d,dx)} \quad\mbox{for all}~~f ~~\mbox{on}~~\mathbb F_q^d,$$
where $C$ is independent of $|\mathbb F_q|.$ \end{theorem}

\begin{proof}First, notice that if the dimension $d$ is two, then the statement of Theorem \ref{newradon} follows immediately from Theorem \ref{newmain11}.
 We therefore assume that $d\geq 3.$  As before, we may assume that $f$ is a non-negative real function and 
\begin{equation}\label{simple1}\sum_{x\in \mathbb F_q^d} [f(x)]^{(d+1)/d}=1.\end{equation}
Moreover, we may assume that the function $f$ is a step function: 
\begin{equation}\label{simple2} f(x)=\sum_{i=0}^{\infty}  2^{-i} E_i(x),\end{equation}
where  the sets $E_i$ are disjoint subsets of $\mathbb F_q^d.$ Notice that (\ref{simple1}) and (\ref{simple2}) imply that 
\begin{equation}\label{simple3} 
\sum_{j=0}^{\infty} 2^{-\frac{(d+1)j}{d}}|E_j|=1 \quad \mbox{and }~~ |E_j|\leq 2^{\frac{(d+1)j}{d}}  ~~\mbox{for all}~~j=0,1,\dots.
\end{equation}
We write $\Pi_{d-1}= H\cup \Theta$ where $H$ and $\Theta$ are defined by
$$ H:=\{w\in \Pi_{d-1}: (0,\dots,0)\notin w\}$$
and 
$$ \Theta:=\{w\in \Pi_{d-1}: (0,\dots,0)\in w\}.$$ It is clear that $H$ and $\Theta$ are disjoint.
Notice that we can identify  $H$ with $ \mathbb F_q^d\setminus\{(0,\dots,0)\}$ in the sense that if $w\in H$, then there exists a unique $w^\prime \in \mathbb F_q^d\setminus\{(0,\dots,0)\}$ such that 
$$ w=\{x\in \mathbb F_q^d: w^\prime\cdot x=1\}.$$
Thus, if $w\in H$, then we may assume that
$$T_{\Pi_{d-1}} f(w)= \frac{1}{|w|} \sum_{x\in \mathbb F_q^d: w^\prime\cdot x=1} f(x) =\frac{1}{q^{d-1}} \sum_{x\in \mathbb F_q^d: w^\prime\cdot x=1} f(x).$$
On the other hand, for a fixed $w\in \Theta,$ there is a unique line passing through the origin, say $L_w$, such that 
$$w=\{x\in \mathbb F_q^d: w^\prime\cdot x=0 ~\mbox{for all}~w^\prime\in L_w\setminus \{(0,\dots,0\}\}.$$
By selecting one specific $w^\prime\in L_w\setminus \{(0,\dots,0)\}$ we can identify $w\in \Theta$ with the specific point $w^\prime\in  L_w\setminus \{(0,\dots,0)\}.$ Throughout the paper, we denote by $S$ the collection of the specific points each of which is chosen from $L_w\setminus \{(0,\dots,0)\}$ for every $w\in \Theta.$ \footnote{In the Euclidean setting, one can consider the set $S$ as a half part of the unit sphere. However, it is not true in general in the finite field setting. For example, if the dimension $d$ is four and $-1\in \mathbb F_q$ is a square number, then the line $l=\{t(i,1, i, 1): t\in \mathbb F_q\}$ does not intersect  the set $\{x\in \mathbb F_q^4: x_1^2+\cdots+x_4^2=1\}.$}
Thus, we also assume that  if $w\in \Theta,$ then
$$ T_{\Pi_{d-1}}f(w)=\frac{1}{q^{d-1}}\sum_{x\in \mathbb F_q^d: w^\prime\cdot x=0}f(x),$$
where $w^\prime \in S.$
Since $\Pi_{d-1}=H\cup \Theta$ and $H\cap \Theta =\emptyset,$ the Radon transform $T_{\Pi_{d-1}}$ can be viewed as
$$ T_{\Pi_{d-1}}f(w)= T_0f(w) +T_1f(w) \quad\mbox{for}~~w\in \Pi_{d-1},$$
where the operators $T_0$ and $T_1$ are defined as
$$ T_{0}f(w)= \frac{\Theta(w)}{q^{d-1}}  \sum_{x\in \mathbb F_q^d:w^\prime\cdot x=0}f(x)$$
and
$$ T_{1}f(w)=\frac{H(w)}{q^{d-1}}  \sum_{x\in \mathbb F_q^d:w^\prime\cdot x=1}f(x).$$
In order to prove Theorem \ref{newradon}, it therefore suffices to show that the following two inequalities hold:
\begin{equation}\label{goal1}
\|T_{0}f\|_{L^{d+1}( \Pi_{d-1}, d\sigma)} \lesssim \|f\|_{L^{\frac{d+1}{d}}(\mathbb F_q^d,dx)} ,
\end{equation}
and 
\begin{equation}\label{goal2}\|T_{1}f\|_{L^{d+1}( \Pi_{d-1}, d\sigma)} \lesssim \|f\|_{L^{\frac{d+1}{d}}(\mathbb F_q^d,dx)} ,
\end{equation}
where the functions $f$ satisfy  (\ref{simple1}), (\ref{simple2}), and  (\ref{simple3}).

\subsection{Proof of the inequality (\ref{goal1})} Let us denote by $\chi$ the canonical additive character of $\mathbb F_q$ (see \cite{LN97} or \cite{IK04}).
Recall that  the orthogonality relation of $\chi$ holds:  
$$\sum_{s\in \mathbb F_q} \chi{(as)}=\left\{\begin{array}{ll} 0 \quad &\mbox{if} ~~a\in \mathbb F_q^*=\mathbb F_q\setminus \{0\}\\
                                                   q\quad & \mbox{if} ~~ a=0. \end{array}\right.$$
                                                                                           
Using the orthogonality relation of $\chi$, we have 
$$ T_0f(w)= \frac{\Theta(w)}{q^{d-1}}  \sum_{x\in \mathbb F_q^d}\left[ \frac{1}{q} \sum_{s\in \mathbb F_q} \chi(s(w^\prime\cdot x))\right] f(x)$$
$$=\frac{\Theta(w)}{q^{d}}  \sum_{x\in \mathbb F_q^d}\sum_{s=0} \chi(s(w^\prime\cdot x))f(x) + \frac{\Theta(w)}{q^{d}}  \sum_{x\in \mathbb F_q^d}\sum_{s\in \mathbb F_q^*} \chi(s(w^\prime\cdot x))f(x)$$
$$= \frac{\Theta(w)}{q^{d}}  \sum_{x\in \mathbb F_q^d}f(x) + \frac{\Theta(w)}{q^{d}}  \sum_{x\in \mathbb F_q^d}\sum_{s\in \mathbb F_q^*} \chi(s(w^\prime\cdot x))f(x)$$
$$:= T_0^\star f(w) + T_0^{\star\star}f(w).$$
Since $|T_0^{\star} f(w)|\leq \|f\|_{L^1(\mathbb F_q^d, dx)}$ for all $w\in \Pi_{d-1},$ we see that
$$ \|T_0^{\star}f \|_{L^{d+1}(\Pi_{d-1}, d\sigma)} \leq  \|f\|_{L^1(\mathbb F_q^d, dx)} \leq \|f\|_{L^{\frac{d+1}{d}}(\mathbb F_q^d,dx)},$$
where we used the facts that $dx$ and $d\sigma$ are the normalized counting measure on $\mathbb F_q^d$ and the normalized surface measure on $\Pi_{d-1}$ respectively. To prove the inequality (\ref{goal1}), it remains to prove that for all functions $f$ satisfying  (\ref{simple1}), (\ref{simple2}), and  (\ref{simple3}), 
\begin{equation}\label{remain1}
\|T_{0}^{\star\star}f\|_{L^{d+1}( \Pi_{d-1}, d\sigma)} \lesssim \|f\|_{L^{\frac{d+1}{d}}(\mathbb F_q^d,dx)} = q^{-\frac{d^2}{d+1}},
\end{equation}
where the last equality follows from (\ref{simple1}).
We need the following lemma.
\begin{lemma}\label{reallem} Let $d\geq 3$. Then, for every subset $E\subset \mathbb F_q^d,$ we have
$$ \|T_{0}^{\star\star} E\|_{L^{\frac{d+1}{2}}(\Pi_{d-1}, d\sigma)} \lesssim q^{-\frac{(d^2+1)}{d+1}} |E|^{\frac{d-1}{d+1}}.$$
\end{lemma}
\begin{proof} Since $d\geq 3,$  we see that the statement of Lemma \ref{reallem} follows immediately by interpolating the following two estimates: for all indicator functions $E(x)$ on $\mathbb F_q^d,$ 
\begin{equation}\label{inter1}\|T_{0}^{\star\star} E\|_{L^{\infty}(\Pi_{d-1}, d\sigma)} \lesssim q^{-d+1} |E| \end{equation}
and 
\begin{equation}\label{inter2}\|T_{0}^{\star\star} E\|_{L^{2}(\Pi_{d-1}, d\sigma)} \lesssim q^{-d+\frac{1}{2}}|E|^{\frac{1}{2}}.\end{equation}
To obtain (\ref{inter1}), notice that 
$$\|T_{0}^{\star\star} E\|_{L^{\infty}(\Pi_{d-1}, d\sigma)} =\| T_0E -T_{0}^{\star} E\|_{L^{\infty}(\Pi_{d-1}, d\sigma)}$$
$$\leq \|T_{0} E\|_{L^{\infty}(\Pi_{d-1}, d\sigma)}+ \|T_{0}^{\star} E\|_{L^{\infty}(\Pi_{d-1}, d\sigma)} \leq  q^{-d+1}|E| +q^{-d}|E| \sim q^{-d+1}|E|.$$
It remains to prove that (\ref{inter2}) holds. It follows that for every set $E\subset \mathbb F_q^d,$
$$ \|T_{0}^{\star\star} E\|^2_{L^{2}(\Pi_{d-1}, d\sigma)}=\frac{1}{|\Pi_{d-1}|} \sum_{w\in \Pi_{d-1}} \left| \frac{\Theta(w)}{q^{d}}  \sum_{x\in \mathbb F_q^d}\sum_{s\in \mathbb F_q^*} \chi(s(w^\prime\cdot x))E(x) \right|^2$$
$$= \frac{1}{|\Pi_{d-1}|} \sum_{w^\prime\in S} \left| \frac{1}{q^{d}}  \sum_{x\in E}\sum_{s\in \mathbb F_q^*} \chi(s(w^\prime\cdot x)) \right|^2 :=\frac{1}{|\Pi_{d-1}|} \sum_{w^\prime\in S} \Gamma(w^\prime),$$
where we used the fact that $S$ can be identified with $\Theta.$
Using a change of variables, we see that for each $w^\prime\in S, ~ \Gamma(w^\prime)=\Gamma(tw^\prime)$ for all $t\in \mathbb F_q^*.$
By the definition of $S,$ it therefore follows that
$$\|T_{0}^{\star\star} E\|^2_{L^{2}(\Pi_{d-1}, d\sigma)}\leq   \frac{1}{|\Pi_{d-1}|} \frac{1}{q-1} \sum_{w^\prime\in \mathbb F_q^d} \left| \frac{1}{q^{d}}  \sum_{x\in E}\sum_{s\in \mathbb F_q^*} \chi(s(w^\prime\cdot x)) \right|^2. $$
Since $|\Pi_{d-1}|\sim q^d$, if we expand the square term and apply the orthogonality relation of $\chi$ to the sum over $w^\prime\in \mathbb F_q^d$, then we see that
$$\|T_{0}^{\star\star} E\|^2_{L^{2}(\Pi_{d-1}, d\sigma)}\lesssim \frac{1}{q^{3d+1}} \sum_{w^\prime\in \mathbb F_q^d} \sum_{x,x^\prime \in E} \sum_{s,s^\prime \in \mathbb  F_q^*} \chi(w^\prime\cdot (sx-s^\prime x^\prime))$$
$$=\frac{1}{q^{2d+1}} \sum_{x,x^\prime \in E, s,s^\prime \in \mathbb F_q^*: sx=s^\prime x^\prime} 1 \leq \frac{|E|}{q^{2d-1}}.$$
Thus, the proof of Lemma \ref{reallem} is complete.
\end{proof}

We now prove  (\ref{remain1}). Since we have assumed that $f$ is considered as a step function (\ref{simple2}), it follows that 
$$\|T_{0}^{\star\star}f\|^2_{L^{d+1}( \Pi_{d-1}, d\sigma)}=\|(T_{0}^{\star\star}f) (T_{0}^{\star\star}f)\|_{L^{\frac{d+1}{2}}( \Pi_{d-1}, d\sigma)} $$
$$\leq \sum_{i=0}^\infty \sum_{j=0}^\infty 2^{-i-j} \|(T_{0}^{\star\star}E_i) (T_{0}^{\star\star}E_j)\|_{L^{\frac{d+1}{2}}( \Pi_{d-1}, d\sigma)}$$
$$\sim \sum_{i=0}^\infty \sum_{j\geq i}^\infty 2^{-i-j} \|(T_{0}^{\star\star}E_i) (T_{0}^{\star\star}E_j)\|_{L^{\frac{d+1}{2}}( \Pi_{d-1}, d\sigma)},$$
where the last line follows from the symmetry of $i,j.$ By H\"{o}lder's inequality, the inequality (\ref{inter1}), and  Lemma \ref{reallem},  (\ref{remain1}) will follow if we prove that
$$ \sum_{i=0}^\infty \sum_{j\geq i}^\infty 2^{-i-j} |E_i||E_j|^{\frac{d-1}{d+1}}\lesssim 1.$$
This can be justified by making use of  (\ref{simple3}) and computing the summation  over $j$ variable:
$$\sum_{i=0}^\infty \sum_{j\geq i}^\infty 2^{-i-j} |E_i||E_j|^{\frac{d-1}{d+1}} \leq \sum_{i=0}^\infty \sum_{j\geq i}^\infty 2^{-i-j} |E_i| \left( 2^{\frac{(d+1)j}{d}} \right)^{\frac{d-1}{d+1}}\sim 
\sum_{i=0}^\infty |E_i|  2^{-\frac{(d+1)i}{d}}=1,$$
which completes the proof of  (\ref{goal1}).

\subsection{Proof of the inequality (\ref{goal2})}
By showing that the inequality (\ref{goal2}) holds, we shall complete the proof of Theorem \ref{newradon}.
We shall take the same steps as in the previous subsection.
From the orthogonality relation of $\chi,$  it follows that for all $w\in \Pi_{d-1},$
$$ T_1f(w)= \frac{H(w)}{q^d} \sum_{x\in \mathbb F_q^d} f(x) + \frac{H(w)}{q^d} \sum_{x\in \mathbb F_q^d} \sum_{s\in \mathbb F_q^*} \chi(s(w^\prime\cdot x-1))f(x) 
:= T_1^{\star}(w)+ T_1^{\star\star}(w).  $$
As before, it is easy to see that 
$$\|T_1^{\star}f \|_{L^{d+1}(\Pi_{d-1}, d\sigma)}  \leq \|f\|_{L^{\frac{d+1}{d}}(\mathbb F_q^d,dx)}.$$
Thus, it is enough to prove that  for all functions $f$ satisfying  (\ref{simple1}), (\ref{simple2}), and  (\ref{simple3}), 
$$
\|T_{1}^{\star\star}f\|_{L^{d+1}( \Pi_{d-1}, d\sigma)} \lesssim \|f\|_{L^{\frac{d+1}{d}}(\mathbb F_q^d,dx)} = q^{-\frac{d^2}{d+1}},
$$
where the last equality follows from (\ref{simple1}).
From the same arguments as in the proof of  (\ref{goal1}), our task is only to obtain Lemma \ref{reallem} for the operator $T_1^{\star\star}.$  As in the proof of Lemma \ref{reallem}, it suffices to prove the following two equalities: for every subset $E$ of $\mathbb F_q^d,$
\begin{equation}\label{inter11}\|T_{1}^{\star\star} E\|_{L^{\infty}(\Pi_{d-1}, d\sigma)} \lesssim q^{-d+1} |E| \end{equation}
and 
\begin{equation}\label{inter22}\|T_{1}^{\star\star} E\|_{L^{2}(\Pi_{d-1}, d\sigma)} \lesssim q^{-d+\frac{1}{2}}|E|^{\frac{1}{2}}.\end{equation}

The inequality (\ref{inter11}) follows immediately from the same argument as  before.
To prove  (\ref{inter22}), we observe that
$$ \|T_{1}^{\star\star} E\|^2_{L^{2}(\Pi_{d-1}, d\sigma)}= \frac{1}{|\Pi_{d-1}|} \sum_{w\in  \Pi_{d-1}} \left|\frac{H(w)}{q^d} \sum_{x\in \mathbb F_q^d} \sum_{s\in \mathbb F_q^*} \chi(s(w^\prime\cdot x-1))E(x) \right|^2 $$
$$= \frac{1}{|\Pi_{d-1}|} \sum_{w\in  H} \left|\frac{1}{q^d} \sum_{x\in E} \sum_{s\in \mathbb F_q^*} \chi(s(w^\prime\cdot x-1)) \right|^2. $$
Recall that  $H \subset \Pi_{d-1}$ can be identified with $\mathbb F_q^d\setminus \{(0,\dots,0)\}.$ Thus, if we dominate the sum over $w\in H$ by the sum over $w^\prime\in \mathbb F_q^d$, expand the square term, and use the orthogonality relation of $\chi$ over the variable $w^\prime\in \mathbb F_q^d,$ then it follows that
$$  0\leq \|T_{1}^{\star\star} E\|^2_{L^{2}(\Pi_{d-1}, d\sigma)}\leq \frac{1}{|\Pi_{d-1}|q^{d}} \sum_{x,x^{\prime} \in E, s,s^{\prime} \in \mathbb F_q^*: sx=s^{\prime} x^{\prime}} \chi(-s+s^{\prime}). $$  $$ =\frac{(q-1)|E|}{|\Pi_{d-1}|q^{d}} + \frac{1}{|\Pi_{d-1}|q^{d}} \sum_{x,x^{\prime} \in E, s,s^{\prime} \in \mathbb F_q^*: sx=s^{\prime} x^{\prime}, s\neq s^{\prime}} \chi(-s+s^{\prime})= \mbox{I} + \mbox{II}. $$
Since $|\Pi_{d-1}|\sim q^d$, it is clear that $\mbox{I}\sim \frac{|E|}{q^{2d-1}}.$ We claim that $ \mbox{II}\leq 0.$ Indeed, if we use a change of the variables by putting $s=t, \frac{s^\prime}{s}=u,$ then we see that
$$ \mbox{II}= \frac{1}{|\Pi_{d-1}|q^{d}} \sum_{x,x^\prime \in E, t,u\in \mathbb F_q^*: x=ux^\prime, u\neq 0,1} \chi(-t(1-u)).$$
Since $u\neq 1,$ the summation over $t\in \mathbb F_q^*$ is exactly $-1.$ Thus, our claim follows from the observation:
$$ \mbox{II}= \frac{-1}{|\Pi_{d-1}|q^{d}}\sum_{x,x^\prime \in E, u\in \mathbb F_q^*: x=ux^\prime, u\neq 0,1} 1 \leq 0.$$ Therefore, we conclude that 

$$\|T_{1}^{\star\star} E\|^2_{L^{2}(\Pi_{d-1}, d\sigma)}\leq \mbox{I} + \mbox{II}\leq \mbox{I} \sim \frac{|E|}{q^{2d-1}},$$
which implies that the inequality (\ref{inter22}) holds . We have finished proving Theorem \ref{newradon}. \end{proof}

\bibliographystyle{amsplain}

\begin{thebibliography}{10}


\bibitem{Ch84} M. Christ, \textit{ Estimates for the $k$-plane transform,} Indiana Univ. Math. J. \textbf{33} (1984), 891--910.

\bibitem{CSW08} A.~Carbery, B.~Stones, and J.~Wright, \textit {Averages in vector spaces over finite fields,} Math. Proc. Camb. Phil. Soc. (2008), 144, 13, 13--27.

\bibitem{Dr84} S.W. Drury,  \textit{Generalizations of Riesz potentials and $L^p$ estimates for certain $k$-plane transforms,} Illinois J. Math. \textbf{28} (1984), 495--512.

\bibitem{EOT} J. S. Ellenberg,~R. Oberlin, and T.~ Tao, \textit{The Kakeya set and maximal conjectures for algebraic varieties over finite fields},  Mathematika, to appear (www.arxiv.org). 

\bibitem{IK04} H.~ Iwaniec, and E.~Kowalski,  \textit{Analytic Number Theory}, Colloquium Publications \textbf{53} (2004).


\bibitem{LL10} A.~ Lewko and M.~Lewko, \emph{Endpoint restriction estimates for the paraboloid over finite fields}, Proc. Amer. Math. Soc. 140 (2012), 2013-2028. 

\bibitem{LN97} R. Lidl and H. Niederreiter,  \textit{Finite Fields,} Cambridge University Press, Cambridge (1997).

\bibitem{OS82} D.M. Oberlin and E. M. Stein,  \textit{Mapping properties of the Radon transform,} Indiana Univ. Math. J. \textbf{31} (1982), 641--650. 

\bibitem{So79} D. Solmon,  \textit{A note on $k$-plane integral transforms,} J. Math. Anal. Appl. \textbf{ 71} (1979), 351--358.

\end{thebibliography}

\end{document}